\documentclass[12pt]{amsart}

\usepackage{amssymb}
\usepackage{url}
\usepackage[all]{xy}
\usepackage{multicol}
\usepackage{amsthm}
\usepackage{amsmath}
\usepackage{amssymb}
\usepackage{geometry}
\usepackage{pinlabel}
\usepackage{hyperref}

\usepackage{enumitem}

\makeatletter
\@namedef{subjclassname@2010}{%
  \textup{2010} Mathematics Subject Classification}
\makeatother

\usepackage[T1]{fontenc}
\newtheorem{theorem}{Theorem}[section]
\newtheorem{corollary}[theorem]{Corollary}
\newtheorem{lemma}[theorem]{Lemma}
\newtheorem{proposition}[theorem]{Proposition}

\newtheorem*{itheorem}{Theorem}
\newtheorem*{iCorollary}{Corollary}
\newtheorem*{iProposition}{Proposition}

\theoremstyle{definition}
\newtheorem{definition}[theorem]{Definition}

\newtheorem{example}[theorem]{Example}

\numberwithin{equation}{section}

\newcommand{\ZZ}{\mathbb {Z}}

\newcommand{\aff}{\textrm{Aff}}

\def \dow {\underline{\ast }}
\def \up {\overline{\ast }}

\begin{document}

%%%%% To ease editing, for IMPAN journals add:

\baselineskip=17pt

%%%%%%%%%%%%%%%%

\title[Constructing biquandles]{Constructing biquandles}

\author[E. Horvat]{Eva Horvat}
\address{University of Ljubljana\\ Faculty of Education\\
Kardeljeva plo\v s\v cad 16\\
1000 Ljubljana, Slovenia}
\email{eva.horvat@pef.uni-lj.si}

\date{\today }

\begin{abstract}
We define biquandle structures on a given quandle, and show that any biquandle is given by some biquandle structure on its underlying quandle. By determining when two biquandle structures yield isomorphic biquandles, we obtain a relationship between the automorphism group of a biquandle and the automorphism group of its underlying quandle. As an application, we determine the automorphism groups of Alexander and dihedral biquandles. We also discuss product biquandles and describe their automorphism groups. 
\end{abstract}



\maketitle

\label{sec0}
Recently we have witnessed an outburst of research on racks, quandles and related structures in connection with knot theory \cite{JOY}, \cite{FR}, \cite{FR1}, \cite{KAUF}. Since then, the algebraic study of quandles and their automorphism groups is well under way. Biquandles, as algebraic generalization of quandles, are not so well known. Study of biquandles began with \cite{KAUF1}, and biquandle invariants have been amply used in the theory of virtual and other knots \cite{CAR}, \cite{ISH}, \cite{ASH}, \cite{EL}. The structure of biquandles is more complicated than the quandle structure, and consequently they are harder to understand. We would like to find a way of constructing new biquandles with a chosen structure. \\

In this paper, we explore the relationship between quandles and biquandles. Our study is based on the functor $\mathcal{Q}$, defined in \cite{ASH}. We define biquandle structures on a given quandle. We show that every biquandle $B$ is given by a biquandle structure on its underlying quandle $\mathcal{Q}(B)$. By determining when two biquandle structures are isomorphic, we are able to characterize all biquandles with a given underlying quandle. Using the knowledge of quandles and their automorphism groups together with our results, one may construct a fair amount of new biquandles.\\

By \cite{FR}, a knot $K$ in any 3-manifold defines a fundamental quandle $Q(K)$ that is a knot invariant. By choosing a suitable biquandle structure (that would depend on the ambient 3-manifold rather than the knot), fundamental quandles may be turned into biquandles, which might be more suitable for some purposes (like studying quantum enhancements \cite{NEL2} or parity of knots \cite{MAN}). \\

An advantage of our construction is that it yields biquandles with a chosen structure. It also lays ground for a theoretical (versus computer - based) knowledge about the number of biquandles of a given order. We obtain the following partial result in this direction: 
\begin{iCorollary}{\rm \textbf{\ref{cor1}}} The number of nonisomorphic constant biquandle structures on a quandle $Q$ is equal to the number of conjugacy classes of $Aut(Q)$.  
\end{iCorollary}
Using the characterization of nonisomorphic biquandle structures, we obtain a relationship between the automorphism group of a given biquandle and the automorphism group of its underlying quandle:
\begin{itheorem}{\rm \textbf{\ref{th3}}} Let $B$ be a biquandle with $\mathcal{Q}(B)=Q$ that is given by a biquandle structure $\{\beta _{y}|\, y\in Q\}\subset Aut(Q)$. Then $$Aut(B)\leq N_{Aut(Q)}\left \{ \beta _{y}|\, y\in Q\right \}\;.$$ 
\end{itheorem}
Moreover, in case of a constant biquandle structure, we obtain
\begin{iCorollary}{\rm \textbf{\ref{th4}}} Let $X_{f}$ be a biquandle with $\mathcal{Q}(X_f)=Q$ that is given by a constant biquandle structure $\{\beta _{y}=f|y\in Q\}$. Then $Aut(X_{f})\cong C_{Aut(Q)}(f)$. 
\end{iCorollary}
In particular, this yields the automorphism groups of Alexander and dihedral biquandles, see Corollary \ref{cor2} and Proposition \ref{prop3}. 

We also study the automorphism group of product biquandles. 
\begin{iProposition}{\rm \textbf{\ref{prop6}}} Let $(Q,*)$ and $(K,\circ )$ be connected quandles, and denote by $B$ their product biquandle. Then  $$Aut(B)\cong Aut(Q)\times Aut(K)\;.$$
\end{iProposition}

The paper is organized as follows. In Section \ref{sec1}, the basic definitions concerning quandles and biquandles are recalled. In Section \ref{sec2}, we define the functor $\mathcal{Q}$ from category of biquandles to the quandle category. Starting from a given quandle $Q$, we impose on $Q$ a biquandle structure to obtain a biquandle $B$ with $\mathcal{Q}(B)=Q$. We show that every biquandle is obtained by such construction, and give some examples. Further, we determine when two biquandle structures are isomorphic, thus giving a characterization of biquandle structures. In Section \ref{sec3}, we apply our results to describe automorphism groups of biquandles. We obtain a relationship between the automorphism group of a biquandle and the automorphism group of its underlying quandle. We determine the automorphism group of biquandles with a constant biquandle structure, which includes the Alexander and dihedral biquandles. In Section \ref{sec4}, we introduce product biquandles and describe their automorphism groups.  

\section{Preliminaries}
\label{sec1}

\begin{definition}\label{def1} A \textbf{quandle} is a set $Q$ with a binary operation $*\colon Q\times Q\to Q$ that satisfies the following axioms:
\begin{enumerate}
\item $x*x=x$ for every $x\in Q$; 
\item the map $S_y\colon Q\to Q$, given by $S_y (x)=x*y$, is a bijection for every $y\in Q$;
\item $(x*y)*z=(x*z)*(y*z)$ for every $x,y,z\in Q$. 
\end{enumerate} 
\end{definition}

In a quandle $(Q,*)$, we will denote by $x*^{-1}y=S_{y}^{-1}(x)$ the unique element $w\in Q$ for which $w*y=x$. 

A map $f\colon Q_{1}\to Q_{2}$ between quandles is called a \textbf{quandle homomorphism} if $f(x*y)=f(x)*f(y)$ for every $x,y\in Q_{1}$. It follows from Definition \ref{def1} that the map $S_{y}$ is in fact an automorphism of the quandle $Q$. We call these automorphisms the \textbf{symmetries} of $Q$. The subgroup $\langle S_{y}|\, y\in Q\rangle \leq Aut(Q)$ is called the \textbf{inner automorphism group} $Inn(Q)$.   

In the following, we recall some examples of quandles:
\begin{itemize}
\item If $G$ is a group, then $a*b=b^{-1}ab$ defines a quandle operation on $G$; the resulting quandle $(G,*)$ is called the \textbf{conjugation quandle}.
\item In any group $G$, the operation, given by $a\circ b=ba^{-1}b$, also defines a quandle.  
\item Define a binary operation on $\ZZ _{n}$ by $i*j=2j-i\mod n$. Then $R_{n}=(\ZZ _{n},*)$ is a quandle, called the \textbf{dihedral quandle}. 
\item Let $\Lambda =\ZZ [t,t^{-1}]$ and let $M$ be a $\Lambda $-module, then $x*y=tx+(1-t)y$ defines a quandle, that is called an \textbf{Alexander quandle}. 
\end{itemize}

\begin{definition}\label{def2} A \textbf{biquandle} is a set $B$ with two binary operations $\dow ,\up \colon B\times B\to B$ that satisfy the following axioms:
\begin{enumerate}
\item $x\dow x=x\up x$ for every $x\in B$; 
\item the maps $\alpha _y, \beta _y \colon B\to B$ and $S\colon B\times B\to B\times B$, given by $\alpha _y (x)=x\dow y$, $\beta _y (x)=x\up y$ and $S(x,y)=(y\up x,x\dow y)$ are bijections for every $y\in B$;
\item the exchange laws \begin{xalignat*}{1}
& (x\dow y)\dow (z\dow y)=(x\dow z)\dow (y\up z)\;,\\
& (x\dow y)\up (z\dow y)=(x\up z)\dow (y\up z)\textrm{ and }\\
& (x\up y)\up (z\up y)=(x\up z)\up (y\dow z)
\end{xalignat*} are valid for every $x,y,z\in B$. 
\end{enumerate} 
\end{definition}

In a biquandle $(B,\dow ,\up )$, we denote $x\dow ^{-1}y=\alpha _{y}^{-1}(x)$ and $x\up ^{-1}y=\beta _{y}^{-1}(x)$. 

Observe that if $\beta _{y}=id$ for every $y\in B$, then $(B,\dow )$ is a quandle - thus biquandles are a generalization of quandles. We would like to describe the precise relationship between the two algebraic structures.

Some examples of biquandles are listed below:
\begin{itemize}
\item Let $G$ be a group. Define two binary operations on $G$ by $a\dow b=b^{-1}a^{-1}b$ and $a\up b=b^{-2}a$. Then $(G,\dow ,\up )$ is a biquandle, called the \textbf{Wada biquandle}. 
\item Define two operations $\dow $ and $\up $ on the set $\ZZ _{n}$ by $i\dow j=(s+1)j-i$ and $i\up j=si$ for some chosen element $s\in \ZZ _{n}^{*}$. Then $B_{n}=(\ZZ _{n},\dow ,\up )$ is a biquandle, called the \textbf{dihedral biquandle}. 
\item Denote $\Lambda =\ZZ [t^{\pm 1},s^{\pm 1}]$ and let $M$ be a $\Lambda $-module, then $x\dow y=tx+(s-t)y$ and $x\up y=sx$ define a biquandle $(M,\dow ,\up )$, that is called an \textbf{Alexander biquandle}. 
\end{itemize}

For convenience of more topologically oriented readers, we briefly recall the knot-theoretical background of (bi)quandles. Let $D_{L}$ be an oriented link diagram of a link $L$. Figure \ref{fig:crossingQ} depicts the \textbf{quandle crossing relation} at a crossing of the diagram $D_{L}$.

\begin{figure}[h]
\labellist
\normalsize \hair 2pt
\pinlabel $x$ at 10 150
\pinlabel $y$ at 10 40
\pinlabel $x*y$ at 200 40
\endlabellist
\begin{center}
\includegraphics[scale=0.30]{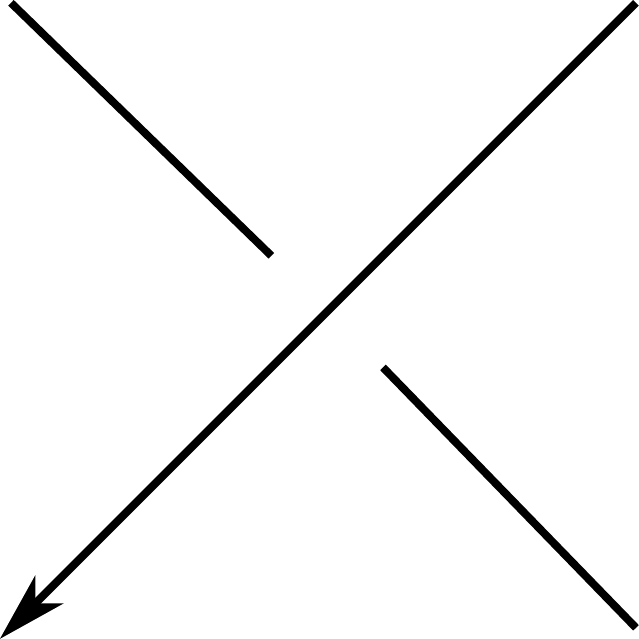}
\caption{The quandle crossing relation}
\label{fig:crossingQ}
\end{center}
\end{figure}

The \textbf{fundamental quandle} of the link $L$ is a quandle, whose generators are the arcs of the diagram and whose relations are the quandle crossing relations. Since the quandle axioms correspond to the Reidemeister moves of a link diagram, the fundamental quandle is a well-defined link invariant. For more details, see \cite{FR}. 

Considering a link diagram $D_{L}$ as a 4-valent graph, every arc becomes divided into two semiarcs. Figure \ref{fig:crossingB} depicts the \textbf{biquandle crossing relations} at a (positive or negative) crossing of the diagram $D_{L}$. The \textbf{fundamental biquandle} of the link $L$ is the biquandle, whose generators are the semiarcs of $D_{L}$, and whose relations are the biquandle crossing relations. Again, the fact that the biquandle axioms correspond to the Reidemester moves ensures that the fundamental biquandle defines a link invariant.

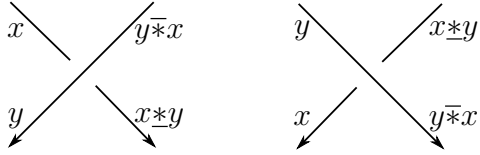
\begin{figure}[h]
\labellist
\normalsize \hair 2pt
\pinlabel $x$ at 10 150
\pinlabel $y$ at 10 40
\pinlabel $x\dow y$ at 190 40
\pinlabel $y\up x$ at 190 150
\pinlabel $y$ at 370 150
\pinlabel $x$ at 370 40
\pinlabel $y\up x$ at 560 40
\pinlabel $x\dow y$ at 560 150
\endlabellist
\begin{center}
\includegraphics[scale=0.30]{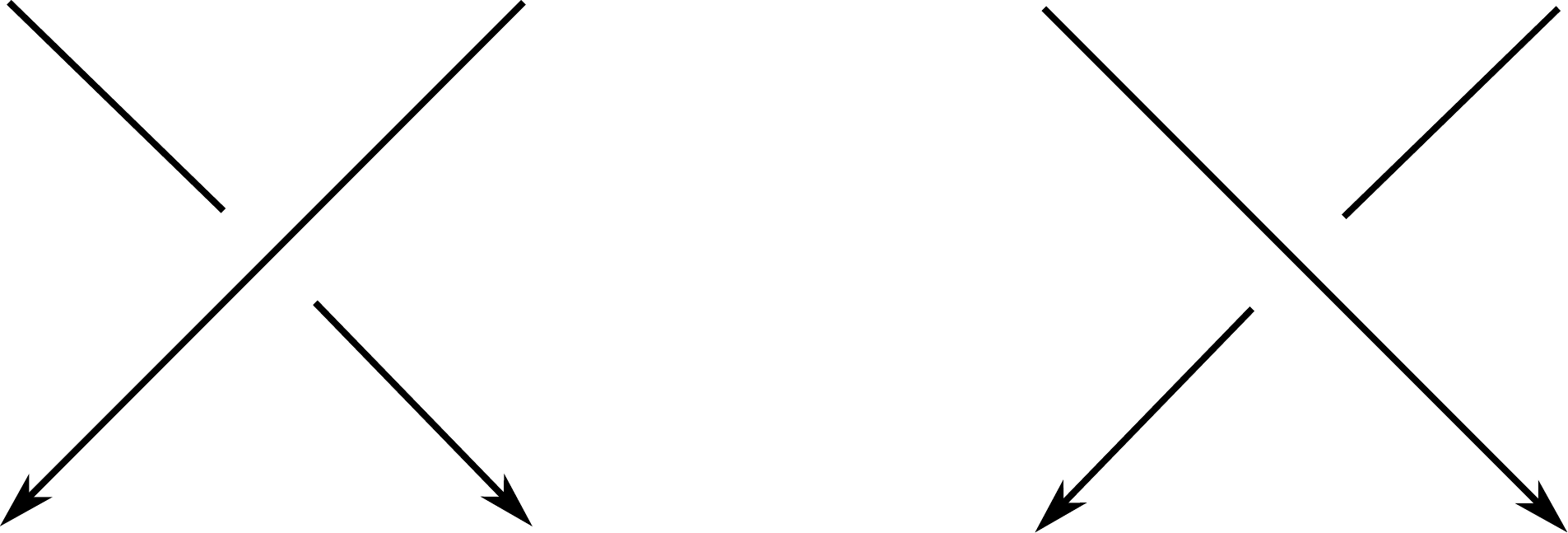}
\caption{Biquandle crossing relations}
\label{fig:crossingB}
\end{center}
\end{figure}

\section{Constructing biquandles from a quandle}
\label{sec2}
 It is known that to any biquandle $B=(X,\dow ,\up )$ we may associate a quandle $\mathcal{Q}(B)=(X,*)$, whose operation is given by $$x*y=(x\dow y)\up ^{-1}y$$ for every $x,y\in X$. In fact, $\mathcal{Q}$ defines a functor from the category of biquandles to the quandle category \cite{ASH}. We reprove this fact in the following Lemma. 

\begin{lemma}\label{lemma1} $\mathcal{Q}$ is a functor from the category of biquandles to the quandle category. 
\end{lemma}
\begin{proof} 
\begin{enumerate}
\item Let $B=(X,\dow ,\up )$ be a biquandle. Biquandle axiom (1) gives the equality $x\dow x=x\up x$, which implies $x*x=x\dow x\up ^{-1}x=x$ for any $x\in X$. Secondly, the maps $\alpha _{y}(x)=x\dow y$ and $\beta _{y}(x)=x\up y$ are bijections, therefore the map $S_{y}\colon \mathcal{Q}(X)\to \mathcal{Q}(X)$, given by $S_{y}(x)=x*y=(\beta _{y}^{-1}\circ \alpha _{y})(x)$, is also a bijection.

To show the validity of the third quandle axiom, we choose $x,y,z\in X$. Denote $b=x*z$, $c=y*z$ and $w=x*y$, which implies $b\up z=x\dow z$, $c\up z=y\dow z$, $w\up y=x\dow y$ and $((x*y)*z)\up z=w\dow z$. We use the third biquandle axiom for $B$ to compute:
\begin{xalignat*}{1}
& \left (((x*z)*(y*z))\up z\right )\up (y\dow z)=\left ((x*z)*(y*z))\up z\right )\up (c\up z)=\\
& =\left ((x*z)*(y*z))\up c\right )\up (z\dow c)=(b\dow c)\up (z\dow c)=(b\up z)\dow (c\up z)=\\
& =(x\dow z)\dow (y\dow z)\\
& \left (((x*y)*z)\up z\right )\up (y\dow z)=(w\dow z)\up (y\dow z)=(w\up y)\dow (z\up y)=\\
& =(x\dow y)\dow (z\up y)=(x\dow z)\dow (y\dow z)\;,
\end{xalignat*} and it follows by the second biquandle axiom that $(x*z)*(y*z)=(x*y)*z$. Therefore $\mathcal{Q}(B)=(X,*)$ is a quandle. 
\item Let $F\colon B_1\to B_2$ be a biquandle homomorphism from a biquandle $B_1=(X,\dow ,\up )$ to another biquandle $B_2=(Y,\veebar ,\barwedge)$. Denote by $\mathcal{Q}(F)\colon \mathcal{Q}(B_{1})\to \mathcal{Q}(B_{2})$ the induced map between the underlying quandles. Then the equalities $$\mathcal{Q}(F)(x*y)=F(x)\dow F(y)\up ^{-1}F(y)=\mathcal{Q}(F)(x)*\mathcal{Q}(F)(y)$$ imply that $\mathcal{Q}(F)$ is a quandle homomorphism. Similarly, for $B_1=B_2$ and $F=id_{B_1}$ we obtain that $\mathcal{Q}(id_{B_1})=id_{\mathcal{Q}(B_1)}$. Also, if $F\colon B_1\to B_2$ and $G\colon B_2\to B_3$ are biquandle homomorphisms, then $\mathcal{Q}(G\circ F)=\mathcal{Q}(G)\circ \mathcal{Q}(F)$. 
\end{enumerate} 
\end{proof}

Thus, every biquandle determines a quandle. This fact rises some questions, like: \begin{itemize}
\item Starting from a quandle $Q$, how do we construct a (nontrivial) biquandle $B$ with $\mathcal{Q}(B)=Q$?
\item Is it possible to characterize all biquandles $B$ with $\mathcal{Q}(B)=Q$?
\item What is the relationship between all biquandles $B$ with $\mathcal{Q}(B)=Q$? When are two of those biquandles isomorphic?
\end{itemize}  

Our results in the remainder of the Section will answer all these questions. We begin by defining a biquandle structure on a given quandle.   
\begin{definition}\label{def3} Let $Q=(X,*)$ be a quandle. A \textbf{biquandle structure} on $Q$ is a family of automorphisms $\{\beta _{y}\colon X\to X|\, y\in X\}\subset Aut(Q)$ that satisfies the following conditions:
\begin{enumerate}
\item $\beta _{\beta _{y}(x*y)}\beta _{y}=\beta _{\beta _{x}(y)}\beta _{x}$ for every $x,y\in X$,
\item the map $(y\mapsto \beta _{y}(y))$ is a bijection of $X$.
\end{enumerate}
\end{definition}
Observe that any automorphism $f$ of a quandle $Q$ defines a biquandle structure on $Q$ by $\beta _{y}=f$ for every $y\in Q$. We call this a \textbf{constant} biquandle structure on $Q$. 
 
\begin{theorem}\label {th1} Let $\{\beta _{y}\colon X\to X|\, y\in X\}$ be a biquandle structure on a quandle $Q=(X,*)$. Define two binary operations on $X$ by $x\dow y=\beta _{y}(x*y)\textrm{ and }x\up y=\beta _{y}(x)$ for every $x,y\in X$. Then $B=(X,\dow ,\up )$ is a biquandle and $\mathcal{Q}(B)=Q$. 
\end{theorem}
\begin{proof} Since $x\dow x=\beta _{x}(x)=x\up x$ for every $x\in X$, the first biquandle axiom is valid.

To verify the second biquandle axiom, observe that since $Q$ is a quandle and $\beta _{y}$ is a bijection, the maps $\alpha _{y}(x)=\beta _{y}(x*y)$ and $\beta _{y}(x)=x\up y$ are bijections for every $y\in X$. It remains to show that the map $S\colon X\times X\to X\times X$, given by $S(x,y)=(y\up x,x\dow y)=(\beta _{x}(y),\beta _{y}(x*y))$, is a bijection. Choose  $(z,w)\in X\times X$. By property (2) from Definition \ref{def3}, there exists a unique $y\in X$ such that $\beta _{w}^{-1}\beta _{z}(z)=\beta _{y}(y)$. There exists a unique $x\in X$ such that $\beta _{y}(x)*\beta _{y}(y)=w$. We have $S(x,y)=(\beta _{x}(y),\beta _{y}(x*y))=(\beta _{x}(y),w)$ and use property (1) in Definition \ref{def3} to calculate
\begin{xalignat*}{1}
& \beta _{\beta _{y}(x*y)}^{-1}\beta _{z}(z)=\beta _{w}^{-1}\beta _{z}(z)=\beta _{y}(y)=\beta _{y}\beta _{x}^{-1}\beta _{x}(y)\\
& \beta _{z}(z)=\beta _{\beta _{y}(x*y)}\beta _{y}\beta _{x}^{-1}\beta _{x}(y)=\beta _{\beta _{x}(y)}\left (\beta _{x}(y)\right )\;,
\end{xalignat*} and property (2) from Definition \ref{def3} implies that $\beta _{x}(y)=z$.

To check the third biquandle axiom, we choose any $x,y,z\in X$ and compute
\begin{xalignat*}{1}
& (x\dow y)\dow (z\dow y)=\beta _{y}(x*y)\dow \beta _{y}(z*y)=\beta _{\beta _{y}(z*y)}(\beta _{y}((x*z)*y))\\
& (x\dow z)\dow (y\up z)=\beta _{z}(x*z)\dow \beta _{z}(y)=\beta _{\beta _{z}(y)}(\beta _{z}((x*z)*y)))\\
& (x\up y)\up (z\up y)=\beta _{y}(x)\up \beta _{y}(z)=\beta _{\beta _{y}(z)}(\beta _{y}(x))\\
& (x\up z)\up (y\dow z)=\beta _{z}(x)\up \beta _{z}(y*z)=\beta _{\beta _{z}(y*z)}(\beta _{z}(x))\\
& (x\dow y)\up (z\dow y)=\beta _{y}(x*y)\up \beta _{y}(z*y)=\beta _{\beta _{y}(z*y)}(\beta _{y}(x*y))\\
&(x\up z)\dow (y\up z) =\beta _{z}(x)\dow \beta _{z}(y)=\beta _{\beta _{z}(y)}(\beta _{z}(x*y))
\end{xalignat*} By condition (1) of Definition \ref{def3}, it follows that the third biquandle axiom is valid. We have shown that $B$ is a biquandle and the equality $x\dow y\up ^{-1}y=\beta _{y}^{-1}(x\dow y)=x*y$ implies that $\mathcal{Q}(B)=Q$. 
\end{proof}

\begin{example}[Wada biquandle] \label{ex3} Let $G$ be a group. It is easy to check that $a*b=ba^{-1}b$ defines a quandle operation on $G$. For every $y\in G$, define a map $\beta _{y}\colon G\to G$ by $\beta _{y}(a)=y^{-2}a$. Since $G$ is a group, $\beta _{y}$ is bijective, and 
\begin{xalignat*}{1}
& \beta _{y}(a)*\beta _{y}(b)=(y^{-2}a)*y^{-2}b=y^{-2}ba^{-1}b=\beta _{y}(a*b)
\end{xalignat*} implies that $\beta _{y}$ is a quandle automorphism of $(G,*)$. The map $(y\mapsto \beta _{y}(y))$ is a bijection since $\beta _{y}(y)=y^{-1}$. We have 
\begin{xalignat*}{1}
& \beta _{\beta _{y}(x*y)}\beta _{y}(a)=(y^{-1}x^{-1}y)^{-2}y^{-2}a=y^{-1}x^{2}y^{-1}a\\
& \beta _{\beta _{x}(y)}\beta _{x}(a)=(x^{-2}y)^{-2}x^{-2}a=y^{-1}x^{2}y^{-1}a
\end{xalignat*} for every $x,y,a\in G$, so the family of automorphisms $\{\beta _{y}|\, y\in G\}$ defines a biquandle structure on $(G,*)$. By Theorem \ref{th1}, this structure defines a biquandle $(G,\dow ,\up )$ with operations $x\dow y=y^{-1}x^{-1}y$ and $x\up y=y^{-2}x$, which is exactly the Wada biquandle. 
\end{example}

Theorem \ref{th1} describes a construction of a biquandle from a given quandle $Q$. In the following Theorem, we show that every biquandle $B$ with $\mathcal{Q}(B)=Q$ is obtained by this construction, thus giving a complete classification of all such biquandles. 

\begin{theorem} \label{th2} Let $B=(X,\dow ,\up )$ be a biquandle and let $\mathcal{Q}(B)=(X,*)$ be its associated quandle. Then the family of mappings $\{\beta _{y}|\, y\in X\}$ is a biquandle structure on $\mathcal{Q}(B)$. 
\end{theorem}
\begin{proof} Since $B$ is a biquandle, the mapping $\beta _{y}\colon X\to X$, given by $\beta _{y}(x)=x\up y$, is a bijection for every $y\in X$. Moreover, we have $x*y=(x\dow y)\up ^{-1}y=\beta _{y}^{-1}(x\dow y)$, which implies $x\dow y=\beta _{y}(x*y)$.   

It follows from the third biquandle axiom that $\beta _{\beta _{y}(z)}(\beta _{y}(x))=\beta _{y}(x)\up \beta _{y}(z)= (x\up y)\up (z\up y)=(x\up z)\up (y\dow z)=\beta _{z}(x)\up \beta _{z}(y*z)=\beta _{\beta _{z}(y*z)}(\beta _{z}(x))$, which implies the equality $\beta _{\beta _{y}(z)}\beta _{y}=\beta _{\beta _{z}(y*z)}\beta _{z}$ for every $y,z\in X$. The maps $\beta _{y}$ thus satisfy condition (1) from Definition \ref{def3} and we may compare
\begin{xalignat*}{1}
& (x\dow y)\up (z\dow y)=\beta _{y}(x*y)\up \beta _{y}(z*y)=\beta _{\beta _{y}(z*y)}(\beta _{y}(x*y))=\beta _{\beta _{z}(y)}(\beta _{z}(x*y))\textrm{ and }\\
& (x\up z)\dow (y\up z)=\beta _{z}(x)\dow \beta _{z}(y)=\beta _{\beta _{z}(y)}(\beta _{z}(x)*\beta _{z}(y))\;.
\end{xalignat*} By the third biquandle axiom, we have $(x\dow y)\up (z\dow y)=(x\up z)\dow (y\up z)$ and therefore $\beta _{z}(x*y)=\beta _{z}(x)*\beta _{z}(y)$ for every $x,y,z\in X$. We have shown that $\beta _{y}\in Aut(X,*)$ for every $y\in X$. 

It remains to prove condition (2) from Definition \ref{def3}. Since $B$ is a biquandle, the map $S\colon X\times X\to X\times X$, given by $S(x,y)=(y\up x,x\dow y)$, is a bijection. It follows that the restriction $S|_{\Delta }\colon \Delta \to \Delta $, given by $S(x,x)=(\beta _{x}(x),\beta _{x}(x))$, is injective, thus $(x\mapsto \beta _{x}(x))$ is injective. To show it is also surjective, choose any $z\in X$. Since $S$ is a bijection, there exist $x,y\in X$ such that $S(x,y)=(z,z)$. It follows that $z=\beta _{x}(y)=\beta _{y}(x*y)$ and by condition (1) we have $$\beta _{\beta _{y}(x*y)}\beta _{y}=\beta _{\beta _{x}(y)}\beta _{x}=\beta _{\beta _{y}(x*y)}\beta _{x}\;,$$ which implies $\beta _{x}=\beta _{y}$ and thus $\beta _{y}(y)=\beta _{x}(y)=z$.   
\end{proof}

\begin{example}[Alexander quandles and biquandles] \label{ex4} Let $\Lambda =\ZZ [t^{\pm 1},s^{\pm 1}]$. Consider an Alexander quandle as a $\Lambda $-module $M$, whose operation is given by $$x*y=(s^{-1}t)x+(1-s^{-1}t)y\;.$$ Taking a constant biquandle structure $\{\beta _{y}\, |\, \beta _{y}(x)=sx\textrm{ for every $x,y\in M$}\}$, we obtain the Alexander biquandle $(M,\dow ,\up )$ with operations $x\dow y=tx+(s-t)y$ and $x\up y=sx$ for every $x,y\in M$. 
\end{example}

In order to classify biquandles, we need to determine when two biquandle structures yield isomorphic biquandles. 

\begin{proposition} \label{prop2} Let $Q_1=(X_{1},*_{1})$ and $Q_2=(X_{2},*_{2})$ be two quandles. Then a biquandle defined by a biquandle structure $\{\beta _{y}^{1}\, |\, y\in X_{1}\}$ on $Q_{1}$ is isomorphic to a biquandle defined by a biquandle structure $\{\beta _{y}^{2}\, |\, y\in X_{2}\}$ on $Q_{2}$ if and only if there is a quandle isomorphism $F\colon Q_{1}\to Q_{2}$ such that for every $y\in X_{1}$, $$F\beta _{y}^{1}=\beta _{F(y)}^{2}F\;.$$
\end{proposition}

\begin{proof} Denote by $B_ {i}=(X_i,\dow _{i},\up _{i})$ the biquandle, defined by the biquandle structure $\{\beta _{y}^{i}\, |\, y\in X_{i}\}$ on $Q_{i}$ for $i=1,2$. \\
$(\Rightarrow )\colon $ Suppose there exists a biquandle isomorphism $f\colon B_{1}\to B_{2}$. It follows from Lemma \ref{lemma1} that $f$ induces an isomorphism between the underlying quandles $(Q_{1},*_{1})$ and $(Q_{2},*_{2})$. Moreover, we have $f(\beta _{y}^{1}(x))=f(x\up _{1}y)=f(x)\up _{2}f(y)=\beta ^{2}_{f(y)}(f(x))$, which implies the equality $f\beta _{y}^{1}=\beta _{f(y)}^{2}f$ for every $y\in Q_{1}$. 

$(\Leftarrow )\colon $ Suppose that there is a quandle isomorphism $F\colon Q_{1}\to Q_{2}$ such that $F\beta _{y}^{1}=\beta _{F(y)}^{2}F$ for every $y\in X_{1}$. It follows that $$F(x\dow _{1}y)=F(\beta _{y}^{1}(x*_{1}y))=\beta _{F(y)}^{2}F(x*_{1}y)=\beta _{F(y)}^{2}(F(x)*_{2}F(y))=F(x)\dow _{2}F(y)$$ and $F(x\up _{1}y)=F(\beta _{y}^{1}(x))=\beta _{F(y)}^{2}(F(x))=F(x)\up _{2}F(y)$ for every $x,y\in X_{1}$, thus $F$ defines a biquandle isomorphism from $B_{1}$ to $B_{2}$. 
\end{proof}

As we have observed, every automorphism $f$ of a quandle $Q=(X,*)$ defines a constant biquandle structure $\{\beta _{y}=f|\, y\in X\}$ and thus defines a biquandle $X_{f}$ with $\mathcal{Q}(X_{f})=Q$. One would then like to know when two such biquandles are isomorphic. 

\begin{corollary} \label{prop1} Let $Q_1=(X,*_{1})$ and $Q_2=(Y,*_{2})$ be quandles and let $f\in Aut(Q_{1})$, $g\in Aut(Q_{2})$ be automorphisms that define biquandles $X_{f}$ and $Y_{g}$. The biquandles $X_{f}$ and $Y_{g}$ are isomorphic if and only if there exists a quandle isomorphism $F\colon Q_{1}\to Q_{2}$ such that $Ff=gF$. 
\end{corollary}
\begin{proof} It follows directly from Proposition \ref{prop2}. 
\end{proof}

\begin{corollary}\label{cor1} The number of nonisomorphic constant biquandle structures on a quandle $Q$ is equal to the number of conjugacy classes of $Aut(Q)$.  
\end{corollary}

\section{Automorphism groups of biquandles}
\label{sec3}
In Section \ref{sec2}, we introduced biquandle structures and showed that every biquandle is given by a biquandle structure on its underlying quandle. Proposition \ref{prop2} determines when two biquandle structures are isomorphic. We may use this result to relate the automorphism group of a biquandle with the automorphism group of its underlying quandle. Given a group $G$ and a subset $S\subset G$, denote by $N_{G}(S)$ the normalizer of $S$ and by $C_{G}(S)$ the centralizer of $S$ in $G$.  

\begin{theorem}\label{th3} Let $B$ be a biquandle with $\mathcal{Q}(B)=Q$ that is given by a biquandle structure $\{\beta _{y}|\, y\in Q\}\subset Aut(Q)$. Then $$Aut(B)\leq N_{Aut(Q)}\left \{ \beta _{y}|\, y\in Q\right \}\;.$$ 
\end{theorem}
\begin{proof} In Proposition \ref{prop2}, we take $Q_{1}=Q_{2}=Q$ and $\beta _{y}^{1}=\beta _{y}^{2}$ for every $y\in Q$. 
\end{proof}

In case of a constant biquandle structure, the biquandle automorphism group is completely determined by the quandle automorphism group: 
\begin{corollary}\label{th4} Let $X_{f}$ be a biquandle with $\mathcal{Q}(X_{f})=Q$ that is given by the constant biquandle structure $\{\beta _{y}=f|y\in Q\}$. Then $Aut(X_{f})\cong C_{Aut(Q)}(f)$. 
\end{corollary}
\begin{proof} If $F\in Aut(Q)$ is a quandle automorphism, then it follows from Proposition \ref{prop2} that $$F\in Aut(X_{f})\quad \Leftrightarrow \quad Ff=f F\;.$$
\end{proof}

The automorphism group of Alexander quandles was determined in \cite{HOU}. Using this result together with Corollary \ref{th4}, we may obtain the automorphism group of any Alexander biquandle. 

\begin{corollary}\label{cor2} Let $M$ be an Alexander biquandle with the corresponding Alexander quandle $\mathcal{Q}(M)$, as defined in Example \ref{ex4}. Then $$Aut(M)\cong C_{Aut(\mathcal{Q}(M))}(s)\;.$$ 
\end{corollary}
 Similarly, we obtain a classification of Alexander biquandles as follows. 

\begin{proposition} Alexander biquandles $M$ and $N$ are isomorphic if and only if there exists an isomorphism of Alexander quandles $F\colon \mathcal{Q}(M)\to \mathcal{Q}(N)$ such that $F(sx)=sF(x)$ for every $x\in M$. 
\end{proposition}
\begin{proof} It follows from Proposition \ref{prop2}. 
\end{proof}

In the following, we study the automorphism group of dihedral biquandles. Recall the definition of the affine group of $\ZZ _{n}$: 
$$\aff (\ZZ _{n})=\{f_{a,b}\colon \ZZ _{n}\to \ZZ _{n}|\, f_{a,b}(i)=ai+b,\, a\in \ZZ _{n}^{*}, b\in \ZZ _{n}\}\;.$$
By \cite[Theorem 2.1]{EL1}, the automorphism group of a dihedral quandle $R_{n}$ is isomorphic to the affine group $\aff (\ZZ _{n})$. The reader may check that the underlying quandle of a dihedral biquandle $\mathcal{Q}(B_{n})$ is a quandle with operation $i*j=(1+s^{-1})j-s^{-1}i$, which is a generalization of the dihedral quandle. Nevertheless, we may show the following.

\begin{proposition} \label{prop3} Let $B_{n}$ be a dihedral biquandle for which $s+1\in \ZZ _{n}^{*}$. Then $Aut(B_{n})\cong C_{\aff (\ZZ _{n})}(s)$.
\end{proposition}
\begin{proof} For any $i\in \ZZ _{n}$ we have $(f_{a,b}\, s)(i)=asi+b$ and $(sf_{a,b})(i)=asi+sb$, which implies $C_{\aff (\ZZ _{n})}(s)=\{f_{a,b}\, |\, a\in \ZZ _{n}^{*},(s-1)b=0\}\leq \aff (\ZZ _{n})$. 

Define a map $\psi \colon C_{\aff (\ZZ _{n})}(s)\to Aut(B_{n})$ by $\psi (f_{a,b})=f_{a,b}$. We may compute
\begin{xalignat*}{1}
& f_{a,b}(i\dow j)=a((s+1)j-i)+b=(s+1)aj-ai+b\\
& f_{a,b}(i)\dow f_{a,b}(j)=(ai+b)\dow (aj+b)=(s+1)aj-ai+sb\\
& f_{a,b}(i\up j)=sai+b\\
& f_{a,b}(i)\up f_{a,b}(j)=sai+sb\textrm{ for every $i,j,a,b\in \ZZ _{n}$.}
\end{xalignat*} If $f_{a,b}\in C_{\aff (\ZZ _{n})}(s)$, we have $(s-1)b=0$ and thus $\psi (f_{a,b})$ is a biquandle automorphism of $B_{n}$. The map $\psi $ is clearly a group homomorphism and $Ker(\psi )=\{f_{1,0}\}=\{1\}$. 

It remains to show that $\psi $ is surjective. Choose any element $g\in Aut(B_{n})$. Since $g$ is a biquandle homomorphism, we have $g((s+1)j-i)=(s+1)g(j)-g(i)$ and $g(si)=sg(i)$ for every $i,j\in \ZZ _{n}$. Define a mapping $h\colon B_n\to B_n$ by $h(j)=g(j)-g(0)$. Since $g(0)=sg(0)$, we may compute $$h((s+1)j-i)=(s+1)g(j)-g(i)-g(0)=(s+1)h(j)-h(i)$$ and $h(si)=g(si)-g(0)=s(g(i)-g(0))=sh(i)$, thus $h$ is a biquandle homomorphism of $B_n$. Moreover, $h(0)=0$ and consequently $h(-i)=-h(i)$ for every $i\in \ZZ _n$. Then we have $h(s+1)=(s+1)h(1)$ and since $h(j(s+1))=h((s+1)-(-(j-1)(s+1)))=(s+1)h(1)+h((j-1)(s+1))$ for $2\leq j\leq n-1$, it follows by induction that $$h(j(s+1))=j(s+1)h(1)$$ for every $j$. By our hypothesis, $s+1\in \ZZ _n^{*}$ and therefore $h(j)=jh(1)$ for every $j\in \ZZ _{n}$. We have shown that $g(j)=jh(1)+g(0)=(g(1)-g(0))j+g(0)$, which implies $g=\psi (f_{g(1)-g(0),g(0)})$. Since $(s-1)g(0)=0$, we have $f_{g(1)-g(0),g(0)}\in C_{\aff (\ZZ _{n})}(s)$ as desired. 
\end{proof}

\section{Product biquandles}
\label{sec4}
In this Section, we study a family of biquandles that naturally arise from any pair of quandles. For two quandles $(Q,*)$ and $(K,\circ )$, define two binary operations $\dow $ and $\up $ on the cartesian product $Q\times K$ by 
\begin{xalignat*}{1}
& (x,a)\dow (y,b)=(x*y,a)\;,\\
& (x,a)\up (y,b)=(x,a\circ b)\;.
\end{xalignat*} 

\begin{proposition} $(Q\times K,\dow, \up )$ is a biquandle for any quandles $(Q,*)$ and $(K,\circ )$. 
\end{proposition}
\begin{proof}  Denote $B=(Q\times K,\dow ,\up )$.\\
(1) Since $(Q,*)$ and $(K,\circ )$ are quandles, we have $(x,a)\dow (x,a)=(x,a)\up (x,a)$ for any $(x,a)\in Q\times K$. \\
(2) For any $(y,b)\in Q\times K$, the maps $\alpha _{(y,b)},\beta _{(y,b)}\colon B\to B$ are given by $\alpha _{(y,b)}(x,a)=(x*y,a)$ and $\beta _{(y,b)}(x,a)=(x,a\circ b)$, so they are bijections. Also the map $S\colon B\times B\to B\times B$, given by $S((x,a),(y,b))=((y,b\circ a),(x*y,a))$, is a bijection. \\
(3) Choose any $(x,a),(y,b),(z,c)\in B$ and compute 
\begin{xalignat*}{1}
& ((x,a)\dow (y,b))\dow ((z,c)\dow (y,b))=((x*y)*(z*y),a)=((x*z)*y,a)=\\
& =((x,a)\dow (z,c))\dow ((y,b)\up (z,c))\\
& ((x,a)\dow (y,b))\up ((z,c)\dow (y,b))=(x*y,a)\circ (z*y,c))=(x*y,a\circ c)=\\
& =((x,a)\up (z,c))\dow ((y,b)\up (z,c))\\
& ((x,a)\up (y,b))\up ((z,c)\up (y,b))=(x,(a\circ b)\circ (c\circ b))=(x,(a\circ c)\circ b)=\\
& =((x,a)\up (z,c))\up ((y,b)\dow (z,c))
\end{xalignat*}
\end{proof}

The biquandle $(Q\times K,\dow ,\up )$ will be called the \textbf{product biquandle} of quandles $(Q,*)$ and $(K,\circ )$. Product biquandles were already considered in \cite{KAM} as a tool to study virtual and twisted links. In the remainder of this Section, we will describe the automorphism group of product biquandles.    

Recall that a quandle $(Q,*)$ is called \textbf{connected} if for every $x,y\in Q$, there exist some elements $z_{1},\ldots ,z_{n}\in Q$ so that $y=((x*z_{1})*z_{2}*\ldots )*z_{n}$. For biquandles, we have an analogous definition:

\begin{definition} In a biquandle $X$, consider the equivalence relation $\sim _{c}$, generated by $x\sim _{c}x\dow y$ and $x\sim _{c}x\up y$ for every $x,y\in X$. The equivalence classes are called \textbf{connected components}, and the biquandle is called \textbf{connected} if there is only one class. 
\end{definition}

\begin{lemma} \label{lemma5} If $(Q,*)$ and $(K,\circ )$ are connected quandles, then their product biquandle $B=(Q\times K,\dow ,\up )$ is connected. 
\end{lemma}
\begin{proof} Choose two elements $(x,a),(y,b)\in B$. Since $Q$ is connected, there exist elements $z_{1},\ldots ,z_{n}$ so that $y=((x*z_{1})*z_{2}*\ldots )*z_{n}$ and since $K$ is connected, there exist elements $c_{1},\ldots ,c_{m}$ so that $b=((a\circ c_{1})\circ c_{2}\circ \ldots )\circ c_{m}$. It follows that 
$$(y,b)=\left ((x,a)\dow (z_{1},a)\dow \ldots \dow (z_{n},a)\right )\up (x,c_{1})\up \ldots \up (x,c_{m})\;,$$ therefore $(y,b)$ is in the same connected component as $(x,a)$. 
\end{proof}

\begin{proposition} \label{prop6} Let $(Q,*)$ and $(K,\circ )$ be connected quandles, and denote by $B=(Q\times K,\dow ,\up )$ their product biquandle. Then  $$Aut(B)\cong Aut(Q)\times Aut(K)\;.$$
\end{proposition}
\begin{proof} Consider the map $\phi \colon Aut(Q)\times Aut(K)\to Aut(B)$ that assigns to a pair of automorphisms $f\in Aut(Q)$ and $g\in Aut(K)$ the map of pairs $f\times g\colon B\to B$. Then
\begin{xalignat*}{1}
&( f\times g)\left ((x,a)\dow (y,b)\right )=(f(x*y),g(a))=(f(x),g(a))\dow (f(y),g(b))=\\
& =(f\times g)(x,a)\dow (f\times g)(y,b)\;,\\
& (f\times g)\left ((x,a)\up (y,b)\right )=(f(x),g(a\circ b))=(f(x),g(a)\circ g(b))=\\
& =(f(x),g(a))\up (f(y),g(b))=(f\times g)(x,a)\up (f\times g)(y,b)\;,
\end{xalignat*} therefore $(f\times g)\in Aut(B)$. It is easy to see that $\phi $ is a group homomorphism and that $Ker (\phi )=\{(id_{Q},id_{K})\}$. 

It remains to show that $\phi $ is surjective. Let $F\in Aut(B)$ be an automorphism of the product biquandle. Denote by $p_{1}\colon Q\times K\to Q$ and $p_{2}\colon Q\times K\to K$ the projection maps and let $p_{1}\circ F=F_{1}\colon B\to Q$ and $p_{2}\circ F=F_{2}\colon B\to K$. Since $F$ is a biquandle automorphism, we have 
\begin{xalignat*}{1}
& F((x,a)\dow (y,b))=F(x*y,a)=F(x,a)\dow F(y,b)=(F_{1}(x,a)*F_{1}(y,b),F_{2}(x,a))\\
& F((x,a)\up (y,b))=F(x,a\circ b)=F(x,a)\up F(y,b)=(F_{1}(x,a),F_{2}(x,a)\circ F_{2}(y,b))
\end{xalignat*} which implies $F_{2}(x*y,a)=F_{2}(x,a)$ and $F_{1}(x,a\circ b)=F_{1}(x,a)$ for every $x,y\in Q$ and $a,b\in K$. Since $K$ is connected, it follows that $F_{1}(x,a)=F_{1}(x,b)$ for every $x\in Q$ and every $a,b\in K$, thus $F_{1}$ is actually defined by a map $f\colon Q\to Q$, where $F_{1}(x,a)=f(x)$. Since $Q$ is connected, it follows that $F_{2}(x,a)=F_{2}(y,a)$ for every $x,y\in Q$ and every $a\in K$, thus $F_{2}$ is defined by a map $g\colon K\to K$, where $F_{2}(x,a)=g(a)$. Moreover, the equalities 
\begin{xalignat*}{1}
& f(x*y)=F_{1}(x*y,a)=F_{1}(x,a)*F_{1}(y,a)=f(x)*f(y)\textrm{ and }\\
& g(a\circ b)=F_{2}(x,a\circ b)=F_{2}(x,a)\circ F_{2}(x,b)=g(a)\circ g(b)
\end{xalignat*} for every $x,y\in Q$ and every $a,b\in K$ imply that $f$ and $g$ are quandle homomorphisms. We have shown that $F=f\times g$ and since $F$ is bijective, it follows that both $f$ and $g$ are also bijective, therefore $F\in Im(\phi )$. 
\end{proof}

What about the automorphisms of product biquandles that are not connected? First we make the following simple observations. 

\begin{lemma}\label{lemma6} Let $X$ and $Y$ be biquandles and let $f\colon X\to Y$ be a biquandle homomorphism. If $x_1$ and $x_2$ are in the same connected component of $X$, then $f(x_1)$ and $f(x_2)$ are in the same connected component of $Y$. 
\end{lemma}
\begin{proof}The equivalence relation $\sim _c$ on $X$ is generated by the equivalences $a\sim _ca\dow b$ (type 1) and $a\sim _ca\up b$ (type 2) for every $a,b\in X$. If $x_1\sim _cx_2$ in $X$, this means that $x_1$ and $x_2$ are connected by a sequence of equivalences of type 1 and 2, and since $f$ is a biquandle homomorphism, $f(x_1)$ and $f(x_2)$ are also connected by the same sequence of equivalences in $Y$. Thus $f(x_1)$ and $f(x_2)$ are in the same connected component of $Y$. 
\end{proof}

\begin{lemma}\label{lemma7} Let $Q$ be a quandle and let $f\in Aut(Q)$ be a quandle automorphism. If $x,y\in Q$ are in the same component of $Q$, then $f(x)$ and $f(y)$ are in the same component of $Q$. 
\end{lemma}
\begin{proof} If $x,y$ are in the same component of $Q$, then $y=((x*z_{1})*z_{2}*\ldots )*z_{n}$ for some elements $z_{1},\ldots ,z_{n}\in Q$, which implies $f(y)=((f(x)*f(z_{1}))*f(z_{2})*\ldots )*f(z_{n})$. 
\end{proof}

Let us briefly analyze the automorphism group of a non-connected quandle. Suppose $(Q,*)$ is a quandle with components $Q_{1},\ldots ,Q_{k}$. By Lemma \ref{lemma7}, the restriction of every automorphism $f\in Aut(Q)$ to the component $Q_{i}$ has $Im(f|_{Q_{i}})\subset Q_{j}$ for some $j$. Since $f$ is a bijection, there exists a permutation $\rho \in S_{k}$ such that $Im(f|_{Q_{i}})=Q_{\rho (i)}$ for $i=1,\ldots ,k$. The automorphism $f$ may thus be written as $$f=\oplus _{i=1}^{k}f_{i}\colon Q_{1}\sqcup \ldots \sqcup Q_{k}\to Q_{\rho (i)}\sqcup \ldots \sqcup Q_{\rho (k)}\;,$$ where $f_{i}\colon Q_{i}\to Q_{\rho (i)}$ is a quandle isomorphism. 

\begin{theorem} \label{th6} Let $(Q,*)$ and $(K,\circ )$ be quandles. Denote by $Q_{1},\ldots ,Q_{k}$ the components of $Q$ and by $K_{1},\ldots ,K_{m}$ the components of $K$. A map $F\colon Q\times K\to Q\times K$ is an automorphism of the product biquandle $B=(Q\times K,\dow ,\up )$ if and only if 
\begin{enumerate}
\item there exist maps $f_{1},\ldots ,f_{m}\colon Q\to Q$ and $g_{1},\ldots ,g_{k}\colon K\to K$ such that $f_{i}|_{Q_{j}}$ and $g_{j}|_{K_{i}}$ is a bijection for $i=1,\ldots ,m$ and $j=1,\ldots ,k$,
\item the equalities $f_{i}(x)*f_{r}(y)=f_{i}(x*y)$ and $g_{j}(a)\circ g_{l}(b)=g_{j}(a\circ b)$ hold for every $i,r\in \{1,\ldots ,m\}$ and every $j,l\in \{1,\ldots ,k\}$, 
\item there exists a bijection $\rho \colon (1,\ldots ,k)\times (1,\ldots ,m)\to (1,\ldots ,k)\times (1,\ldots ,m)$ such that $f_{i}(Q_{j})\times g_{j}(K_{i})=Q_{\rho (j,i)_{1}}\times K_{\rho (j,i)_{2}}$ for every $(i,j)\in (1,\ldots ,k)\times (1,\ldots ,m)$ \footnote{For a bijection $\rho \colon (1,\ldots ,k)\times (1,\ldots ,m)\to (1,\ldots ,k)\times (1,\ldots ,m)$, we denote by $\rho (j,i)_1$ and $\rho (j,i)_2$ the first and the second component of the pair $\rho (j,i)$ respectively.}
\end{enumerate} and $F(x,a)=(f_{i}(x),g_{j}(a))$ for every $(x,a)\in Q_{j}\times K_{i}$. 
\end{theorem}
\begin{proof}$(\Rightarrow )$ Suppose $F\colon Q\times K\to Q\times K$ is an automorphism of the product biquandle. Denote $F_{1}=p_{1}\circ F$ and $F_{2}=p_{2}\circ F$, where $p_{1},p_{2}$ are projections to the respective factors of $Q\times K$. Since $F$ is a biquandle automorphism, we have \begin{xalignat*}{1}
& F((x,a)\dow (y,b))=F(x*y,a)=F(x,a)\dow F(y,b)=(F_{1}(x,a)*F_{1}(y,b),F_{2}(x,a))\\
& F((x,a)\up (y,b))=F(x,a\circ b)=F(x,a)\up F(y,b)=(F_{1}(x,a),F_{2}(x,a)\circ F_{2}(y,b))\;,
\end{xalignat*} which implies $F_{2}(x*y,a)=F_{2}(x,a)$ and $F_{1}(x,a\circ b)=F_{1}(x,a)$ for every $x,y\in Q$ and $a,b\in K$. It follows that $F_{1}|_{Q\times K_{i}}(x,a)=f_{i}(x)$ for some map $f_{i}\colon Q\to Q$ and $F_{2}|_{Q_{j}\times K}(x,a)=g_{j}(a)$ for some map $g_{j}\colon K\to K$. It also follows from the above equalities that 
 \begin{xalignat*}{1}
& f_{i}(x)*f_{r}(y)=F_{1}(x,a)*F_{1}(y,b)=F_{1}(x*y,a)=f_{i}(x*y)\\
& g_{j}(a)\circ g_{l}(b)=F_{2}(x,a)\circ F_{2}(y,b)=F_{2}(x,a\circ b)=g_{j}(a\circ b)
\end{xalignat*}
for $i,r\in \{1,\ldots ,m\}$ and $j,l\in \{1,\ldots ,k\}$. We have $F|_{Q_{j}\times K_{i}}=f_{i}\times g_{j}$. 

By Lemma \ref{lemma5}, the sets $Q_{j}\times K_{i}$ are connected components of the product biquandle $B$ and by Lemma \ref{lemma6} we have $Im(F|_{Q_{j}\times K_{i}})\subset Q_{k}\times K_{l}$ for some $k$ and $l$. Since $F$ is an isomorphism, there exists a bijection $\rho \colon (1,\ldots ,k)\times (1,\ldots ,m)\to (1,\ldots ,k)\times (1,\ldots ,m)$ such that $F(Q_{j}\times K_{i})=f_{i}(Q_{j})\times g_{j}(K_{i})=Q_{\rho (j,i)_{1}}\times K_{\rho (j,i)_{2}}$ for $j=1,\ldots ,k$ and $i=1\ldots ,m$.  

Since $F$ is injective, also $F|_{Q_{j}\times K_{i}}=(f_{i}\times g_{j})|_{Q_{j}\times K_{i}}$ is injective. Since $F$ is surjective, it follows from Lemma \ref{lemma6} that also $(f_{i}\times g_{j})|_{Q_{j}\times K_{i}}$ is surjective. Therefore both $f_{i}|_{Q_{j}}$ and $g_{j}|_{K_{i}}$ are bijections for $j=1,\ldots ,k$ and $i=1,\ldots ,m$.  

$(\Leftarrow )$ Suppose that $f_{1},\ldots ,f_{m}\colon Q\to Q$ and $g_{1},\ldots ,g_{k}\colon K\to K$ are maps that satisfy conditions (1), (2) and (3) of the Theorem. Define a map $F\colon Q\times K\to Q\times K$ by $F(x,a)=(f_{i}(x),g_{j}(a))$ for $(x,a)\in Q_{j}\times K_{i}$. Choose two elements $(x,a)\in Q_{j}\times K_{i}$ and $(y,b)\in Q_{l}\times K_{r}$ and compute
\begin{xalignat*}{1}
& F\left ((x,a)\dow (y,b)\right )=F(x*y,a)=(f_{i}(x*y),g_{j}(a))=(f_{i}(x)*f_{r}(y),g_{j}(a))=\\
& =(f_{i}(x),g_{j}(a))\dow (f_{r}(y),g_{l}(b))=F(x,a)\dow F(y,b)\textrm{ and }\\
& F\left ((x,a)\up (y,b)\right )=F(x,a\circ b)=(f_{i}(x),g_{j}(a\circ b))=(f_{i}(x),g_{j}(a)\circ g_{l}(b))=\\
& =(f_{i}(x),g_{j}(a))\up (f_{r}(y),g_{l}(b))=F(x,a)\up F(y,b)\;,
\end{xalignat*} which shows that $F$ is a biquandle homomorphism of the product biquandle $B$. It follows from (3) that if $(i,j)\neq (k,l)$, then $Im(F|_{Q_{i}\times K_{j}})\cap Im(F|_{Q_{k}\times K_{l}})=\emptyset $ and since $F|_{Q_{i}\times K_{j}}$ is injective, then $F$ is injective. Since $F|_{Q_{j}\times K_{i}}$ is surjective for every $(i,j)\in (1,\ldots ,k)\times (1,\ldots ,m)$, it follows from (3) that $F$ is surjective. We have thus shown that $F\in Aut(B)$. 
\end{proof}

Conditions (1) - (3) of Theorem \ref{th6} imply that $f_{i}\in Aut(Q)$ and $g_{j}\in Aut(K)$ for $i=1,\ldots ,m$ and $j=1,\ldots ,k$. Every automorphism of a product biquandle $B=(Q\times K,\dow ,\up )$ is thus given by 
$$F=\oplus _{j=1}^{k}\oplus _{i=1}^{m}(f_{i}\times g_{j})\colon \sqcup (Q_{j}\times K_{i})\to \sqcup (Q_{j}\times K_{i})\;,$$
where $(f_{1},\ldots ,f_{m})$ and $(g_{1},\ldots ,g_{k})$ are tuples of quandle automorphisms of $Q$ and $K$ that are connected by conditions (2) and (3) of Theorem \ref{th6}.

\subsection*{Acknowledgements} This research was supported by the Slovenian Research Agency grant N1-0083.

%%%%%%%%%%% To ease editing, use normal size for the references:

\end{document}